\documentclass[11pt]{amsart}
\usepackage{microtype}
\usepackage[T1]{fontenc}
\usepackage{amsmath, amscd, amsthm} 
\usepackage{amssymb, amsfonts} 
\usepackage{mathtools}
\usepackage{stmaryrd}
\usepackage{enumerate}
\usepackage[svgnames, table, xcdraw]{xcolor} 
\usepackage{color}
\usepackage[margin=1.5in]{geometry}
\usepackage{booktabs}
\usepackage{aliascnt}
\usepackage{hyperref}
 \definecolor{dark-red}{rgb}{0.4,0.15,0.15}
\hypersetup{
    colorlinks, linkcolor=dark-red,
    citecolor=DarkBlue, urlcolor=MediumBlue
}
\usepackage{cancel}

\usepackage{diagbox}

\usepackage{mathrsfs}
\usepackage[
textwidth=3cm,
textsize=small,
colorinlistoftodos]
{todonotes}

\usepackage[all]{xy}

\usepackage{tikz}
\usepackage{tikz-cd}
\usetikzlibrary{decorations.pathreplacing}

\usepackage{pgfplots}

\newcommand{\RR}{\mathcal{R}}

\DeclareMathAlphabet{\mathcal}{OMS}{cmsy}{m}{n}

\newcommand{\Sub}{\operatorname{Sub}}

\usepackage[T1]{fontenc}

\numberwithin{equation}{section} 

\theoremstyle{plain}

\newaliascnt{theorem}{equation}  
\newtheorem{theorem}[theorem]{Theorem}  
\aliascntresetthe{theorem}

\newaliascnt{dodeca}{equation}  
  
\aliascntresetthe{dodeca}

 \theoremstyle{definition}

\newaliascnt{prop}{equation}  
\newtheorem{prop}[prop]{Proposition}
\aliascntresetthe{prop}

\newaliascnt{lemma}{equation}  
\newtheorem{lemma}[lemma]{Lemma}
\aliascntresetthe{lemma}

\newaliascnt{corollary}{equation}  
\newtheorem{corollary}[corollary]{Corollary}
\aliascntresetthe{corollary}

\newaliascnt{claim}{equation}  

\aliascntresetthe{claim}

\newaliascnt{conjecture}{equation}  
\newtheorem{conjecture}[conjecture]{Conjecture}
\aliascntresetthe{conjecture}

\newaliascnt{question}{equation}  

\aliascntresetthe{question}

\newaliascnt{defn}{equation}  
\newtheorem{defn}[defn]{Definition}
\aliascntresetthe{defn}

\newaliascnt{example}{equation}  
\newtheorem{example}[example]{Example}
\aliascntresetthe{example}

\newaliascnt{computation}{equation}  
\newtheorem{computation}[computation]{Computation}
\aliascntresetthe{computation}

\newaliascnt{construction}{equation}  
\newtheorem{construction}[construction]{Construction}
\aliascntresetthe{construction}

\theoremstyle{remark}

\newaliascnt{remark}{equation}  
\newtheorem{remark}[remark]{Remark}
\aliascntresetthe{remark}

\newaliascnt{convention}{equation}  

\aliascntresetthe{convention}

\theoremstyle{plain}

\definecolor{gBlue}{HTML}{2b83ba}
\definecolor{gRed}{HTML}{ff5100}

\begin{document}
\title{The combinatorics of $N_\infty$ operads for $C_{qp^n}$ and $D_{p^n}$}

\author{Scott Balchin}
\address{Queen's University Belfast}
\email{s.balchin@qub.ac.uk}

\author{Ethan MacBrough}
\address{Reed College}
\email{emacbrough@reed.edu}

\author{Kyle Ormsby}
\address{Reed College / University of Washington}
\email{ormsbyk@reed.edu \textnormal{/} ormsbyk@uw.edu}

\begin{abstract}
We provide a general recursive method for constructing transfer systems on finite lattices. Using this we calculate the number of homotopically distinct $N_\infty$ operads for dihedral groups $D_{p^n}$, $p > 2$ prime, and cyclic groups $C_{qp^n}$, $p \neq q$ prime. We then further display some of the beautiful combinatorics obtained by restricting to certain homotopically meaningful $N_\infty$ operads for these groups.
\end{abstract}

\maketitle

\setcounter{tocdepth}{1}
\tableofcontents

\section{Introduction}

Throughout, $D_{p^n}$ denotes the dihedral group of order $2p^n$, $p>2$ prime, and $C_{qp^n}$ the cyclic group of order $qp^n$, for $q,p$ distinct primes. (We allow $p=2$ in the cyclic case.)

This paper --- at its heart --- is concerned with the combinatorial data arising when studying commutative $G$-equivariant spectra for $G$ either $D_{p^n}$ or $C_{qp^n}$. In particular, we wish to study the set of $N_\infty$ operads for these groups, as introduced by Blumberg--Hill~\cite{BlumbergHill}, which govern the ways in which the commutativity respects the group action.

Such an exploration of $N_\infty$ operads for $G= C_{p^n}$ was undertaken in~\cite{bbr}, where it was proved that the collection of these operads (up to homotopy) was in bijection with the Tamari lattice, and in particular, there are $\mathrm{Cat}(n+1)$ many of them, where $\mathrm{Cat}(k)$ is the $k$-th Catalan number. This was achieved using an explicit representation of $N_\infty$ operads as combinatorial structures called \emph{transfer systems} on the subgroup lattice of $C_{p^n}$. As such, one is led to study transfer systems on arbitrary finite lattices. We refer the reader to \cite[\S 4]{fooqw} and \autoref{defn:cattran} for definitions in this context.

Note that the subgroup lattice for $C_{p^n}$ is isomorphic to $[n] = \{0<1<\cdots<n\}$. A key method employed in~\cite{bbr} was a recursive method for building all transfer systems on $[n]$ from ones on $[i]$ where $i < n$. It is this powerful point of view that we will generalize in this paper. In particular, in \autoref{sec:strategy} we shall exhibit an extremely general method of iteratively building all transfer systems on a finite lattice $L$.

We will then implement the algorithm of~\autoref{sec:strategy} in the case of $L = [1] \times [n]$. From a homotopical viewpoint, we observe that $[1] \times [n] \cong \Sub(C_{qp^n})$, and as such, transfer systems on $[1] \times [n]$ correspond to $N_\infty$ operads for $C_{qp^n}$. To do so, we will first begin by considering a restricted collection of transfer systems, namely the \emph{liftable} transfer systems (\autoref{defn:liftable}). From the work of the authors in~\cite{bmometa}, liftable transfer systems on $[1] \times [n]$ are in bijection with $N_\infty$ operads for $D_{p^n}$ (provided $p \neq 2$). A recursive formula for computing the number of liftable transfer systems appears as the main result in \autoref{thm:liftable}.

We are then able to exploit the self duality of transfer systems on $[1] \times [n]$ as observed in~\cite[Theorem 4.21]{fooqw} to complete the computation for $C_{qp^n}$ almost immediately from the case of $D_{p^n}$. This leads to \autoref{thm:nonliftable} where we provide a recursive formula for computing the number of transfer systems on $[1] \times [n]$.\footnote{\texttt{Python}  code for computing the number of $N_\infty$ operads for $C_{qp^n}$ and $D_{p^n}$ using the results of this paper can be found at
\url{https://github.com/bifibrant/recursion/}.}  Thus, we obtain a (recursive) formula for a second and third infinite family of groups following the work of \cite{bbr}\footnote{Since the writing of this paper, a formula for transfer systems for a fourth infinite family of groups, namely groups of the form $C_p \times C_p$ for $p$ prime, was obtained in \cite{bao2023transfer}}.

Following these enumerative results, we begin to explore some of the combinatorial structures appearing in the recursions that we have provided. In \autoref{sec:restrict} we study the notion of \emph{restricted Tamari intervals}, which form an important part of the recursions for $C_{qp^n}$ and $D_{p^n}$, and see that this is related to certain triangulations of polygons.

Finally, in \autoref{sec:max1} we will restrict ourselves to what we call \emph{maximally extendable} transfer systems. In terms of the lattice $[1] \times [n]$, these are those transfer systems whose restriction to the bottom row is the maximal transfer system on $[n]$. We shall see that in the liftable case these have specific combinatorial interpretation (in terms of large Schr\"oder numbers), and we make a conjecture relating the general case to rooted subtrees of a rooted planar tree. Returning to the world of homotopy theory, these maximally extendable transfer systems have a meaninful interpretation. Indeed, they are precisely those transfer systems cooresponding to $N_\infty$ operads which restrict to the genuine $G$-$E_\infty$ operad on $C_{p^n}$.

\subsection*{Acknowledgements}

The authors thank Mike Hill for suggesting the construction presented in \autoref{sec:strategy}. They also thank Ang\'elica Osorno and the anonymous referees for helpful feedback.

The first author would like to thank the Max Planck Institute for Mathematics for its hospitality, and was partially supported by the European Research Council (ERC) under Horizon Europe (grant No. 101042990). The second author thanks Coil Technologies for their generous donation to fund his tuition, which enabled him to conduct this research. The third author's work was supported by the National Science Foundation under Grant No.~DMS-2204365.

Competing interests: the authors declare none.

\section{Transfer systems and a generalized $\odot$ construction}\label{sec:strategy}

We begin by recalling the definition of a transfer system on a poset from~\cite{bmometa,fooqw}.

	\begin{defn}\label{defn:cattran}
Let $\mathcal{P} = (\mathcal{P}, \leqslant)$ be a poset. A (categorical) \emph{transfer system} on $\mathcal{P}$ consists of a partial order $\RR$ on $\mathcal{P}$ that refines $\leqslant$ and such that whenever $x \, \RR \, y$ and $z\leqslant y$, then for all maximal $w\in x^{\downarrow}\cap z^{\downarrow}$ we have $w \, \RR \, z$ where $x^{\downarrow}$ denotes the set of all $y \leqslant x$.
\end{defn}

For a lattice $L$, we write $\mathsf{Tr}(L)$ for the collection of all transfer systems of $L$. In~\cite{bbr} it was proved that $|\mathsf{Tr}([n])|$ coincides with the $(n+1)$-th Catalan number. The proof for this proceeded in a recursive fashion which was governed by an operation denoted $\odot$ in loc.~cit.

Let $\RR$ be a transfer system on $[n]$. Denote by $x$ the minimal element such that $x\, \RR\, n$, and consider the partition of $[n]$ as $[0,x-1] \amalg [x,n]$. The key observation is that the restriction of $\RR$ to $[0,x-1]$ and $[x,n]$ yields two transfer systems, and moreover, $\RR$ can be recovered from the disjoint union of these two transfer systems. To retrieve the aforementioned $\odot$ operation from this, we note that the \emph{pivot element} in~\cite{bbr} is simply $x$.

In this section we will follow a suggestion of Mike Hill to extend this observation to transfer systems on an arbitrary finite lattice $L$, which provides a method for producing recursive formulas for counting transfer systems. To this end, let $L$ be a finite lattice equipped with a transfer system $\RR$. Let $m$ be the maximal element of $L$, we will say that $x \in L$ is \emph{fibrant} if $x\, \RR\, m$. (This terminology is inspired by the corresponding statement under the equivalence of transfer systems to weak factorization systems.)

\begin{lemma}
Let $L$ and $\RR$ be as above. Then $\RR$ has a unique minimal fibrant element.
\end{lemma}

\begin{proof}
Let $\{x_i\}_i$ be the collection of fibrant elements of $\RR$. Then by the definition of transfer systems, it follows that $(\bigwedge_i x_i)\, \RR\, m$. In particular, $\bigwedge_i x_i$ is a fibrant element and $(\bigwedge_i x_i)  \leqslant x_j$ for all $j$ as required.
\end{proof}

\begin{defn}
Let $L$ be a lattice, and $x \in L$. We denote by $x_{\uparrow}$ the set of all $y \geqslant x$. The set theoretic complement of $x_{\uparrow}$ in $L$ will be denoted $x_{\uparrow}^\mathsf{c}$.
\end{defn}

The following is a standard result in the theory of finite lattices.

\begin{lemma}\label{lem:restriction}
Let $L$ be a finite lattice and $x \in L$. Then the collection $x_{\uparrow}$ is a sublattice of $L$, and $x_{\uparrow}^\mathsf{c}$ is a sub-meet-semilattice of $L$.
\end{lemma}

Via \autoref{lem:restriction}, one sees that there is a well-defined notion of transfer system for $x_{\uparrow} $ and $x_{\uparrow}^\mathsf{c}$. Indeed, the axioms for a categorical transfer system only depend on the meet operation, so we can use the same definition for meet-semilattices like $x_\uparrow^{\mathsf{c}}$.  In particular, for $L$ a finite lattice and $\RR$ a transfer system on $L$, if $x$ is the unique minimal fibrant element, then $\RR$ restricts to transfer systems $\RR_2$ and $\RR_1$ on $x_{\uparrow}$ and $x_{\uparrow}^\mathsf{c}$, respectively. 

Given $\RR_1$ and $\RR_2$ as above, we can construct a candidate transfer system $\RR_1 \amalg \RR_2$ by simply taking their disjoint union over $L$. The following lemma proves that this retrieves the starting transfer system $\RR$.

\begin{lemma}
Let $\RR$, $\RR_1$, and $\RR_2$ be as above. Then $\RR = \RR_1 \amalg \RR_2$.
\end{lemma}

\begin{proof}
Assume that there is some $y \in x_\uparrow^{\mathsf{c}}$ and $z \in x_\uparrow$ such that $y\, \RR\, z$. By the property of a transfer system we have $(x \wedge y )\, \RR\, x$, and by transitivity we have $(x \wedge y )\, \RR\, x \RR\, m$, that is, $x \wedge y $ is fibrant. However, as $y \in x_\uparrow^{\mathsf{c}}$, $(x \wedge y ) < x$, violating the assumed minimality of $x$.
\end{proof}

\begin{example}\label{ex:1x3trans}
Let $L = [1] \times [3]$, and consider the following transfer system.\footnote{\textbf{{\fontencoding{U}\fontfamily{futs}\selectfont\char 66\relax} Warning}: For typographical reasons, we plot the first coordinate vertically and the second coordinate horizontally; this same convention is held in all our diagrams and nomenclature (especially \emph{bottom row}, \emph{top row}, \emph{verticals}, and \emph{diagonals}.}
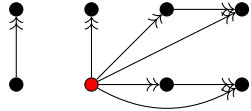
\[
\begin{tikzpicture}
		\node[circle,draw=black, fill=black, inner sep=0pt, minimum size=5pt] (00) at (0,0) {};
		\node[circle,draw=black, fill=red, inner sep=0pt, minimum size=5pt] (10) at (1,0) {};
		\node[circle,draw=black, fill=black, inner sep=0pt, minimum size=5pt] (01) at (0,1) {};
		\node[circle,draw=black, fill=black, inner sep=0pt, minimum size=5pt] (11) at (1,1) {};
		\node[circle,draw=black, fill=black, inner sep=0pt, minimum size=5pt] (20) at (2,0) {};
		\node[circle,draw=black, fill=black, inner sep=0pt, minimum size=5pt] (21) at (2,1) {};
		\node[circle,draw=black, fill=black, inner sep=0pt, minimum size=5pt] (30) at (3,0) {};
		\node[circle,draw=black, fill=black, inner sep=0pt, minimum size=5pt] (31) at (3,1) {};
		\draw[->>] (00) -- (01);
		\draw[->>] (10) -- (11);
		\draw[->>] (20) -- (30);
		\draw[->>] (21) -- (31);
		\draw[->>] (10) -- (20);
		\draw[->>] (10) to [bend right] (30);
		\draw[->>] (10) -- (21);
		\draw[->>] (10) -- (31);
		\end{tikzpicture}
\]
The minimal fibrant element $(0,1)$ is highlighted in red. One clearly sees pictorially that $\RR$ splits as a transfer system $\RR_2$ on $(0,1)_{\uparrow} \cong [1] \times [2]$ and a transfer system $\RR_1$ on $(0,1)_{\uparrow}^{\mathsf{c}} \cong [1] \times [0] \cong [1]$.
\end{example}

As such, every transfer system $\RR$ on $L$ can be split into a pair of transfer systems $\RR_1, \RR_2$ on sub-(semi-)lattices $x_\uparrow^{\mathsf{c}}$ and $x_\uparrow$ of $L$ where $\RR_2$ has minimal fibrant element $x$.  In the converse direction, given two transfer systems $\RR_1$ and $\RR_2$ on $x_\uparrow^\mathsf{c}$ and $x_\uparrow$ respectively, we can ask when $\RR_1 \amalg \RR_2$ is a transfer system for $L$. This happens if and only if the relations in $\RR_2$ when restricted to $x_{\uparrow}^\mathsf{c}$ (i.e., those relations that are forced by the relations of $\RR_2$ under the restriction property of a transfer system on the larger lattice $L$) are relations in $\RR_1$. We will say that such an $\RR_1$ and $\RR_2$ is a \emph{restriction closed pair} over $L$.

\begin{example}
If we modify \autoref{ex:1x3trans} so that $\RR_1$ is the empty transfer system on $[1]$, then this pair is not restriction closed and $\RR_1 \amalg \RR_2$ does not form a transfer system for $[1] \times [3]$.
\end{example}

From this discussion, we can form the basis of a recursive formula for computing $\mathsf{Tr}(L)$ based on the geometry of $L$. We summarise this in the following theorem.

\begin{theorem}
Let $L$ be a finite lattice. Then there is a bijection between transfer systems on $L$ and triples $(x,\RR_1,\RR_2)$ where $x\in L$, $\RR_1$ is a transfer system on $x_\uparrow^{\mathsf c}$, $\RR_2$ is a transfer system on $x_\uparrow$ with minimal fibrant element $x$, and $\RR_1,\RR_2$ form a restriction closed pair.
\end{theorem}

\begin{remark}
The condition that $\RR_2$ has minimal fibrant element $x$ is equivalent to $\RR_2$ being connected when its relations are considered as edges in an undirected graph.
\end{remark}

\begin{construction}
To rephrase this strategy in the form of an $\odot$ operation, we use the observation from the beginning of this section that the \emph{pivot element} of the $\odot$ operation for $[n]$ was given by the minimal fibrant element of $\RR$. If $L$ is now an arbitrary finite lattice with $\RR$, $\RR_1$ and $\RR_2$ as above, we let $\RR_2 \smallsetminus \{x\}$ be the collection of relations where we have removed the initial element (i.e., we remove the fibrant element $x$). Then we could sensibly define $\RR = \RR_1 \odot (\RR_2 \smallsetminus \{x\}$), where the operation $\odot$ inserts a new element $x$ and adds in the relation $x\, \RR\, m$ and all relations induced by this.

In this notation, one would decompose the $\RR$ of \autoref{ex:1x3trans} as\[
\begin{tikzpicture}
		\node[circle,draw=black, fill=black, inner sep=0pt, minimum size=5pt] (00) at (0,0) {};
		\node[circle,draw=black, fill=black, inner sep=0pt, minimum size=5pt] (01) at (0,1) {};
		\node[circle,draw=black, fill=black, inner sep=0pt, minimum size=5pt] (11) at (2,1) {};
		\node[circle,draw=black, fill=black, inner sep=0pt, minimum size=5pt] (20) at (3,0) {};
		\node[circle,draw=black, fill=black, inner sep=0pt, minimum size=5pt] (21) at (3,1) {};
		\node[circle,draw=black, fill=black, inner sep=0pt, minimum size=5pt] (30) at (4,0) {};
		\node[circle,draw=black, fill=black, inner sep=0pt, minimum size=5pt] (31) at (4,1) {};
		\draw[->>] (00) -- (01);
		\draw[->>] (20) -- (30);
		\draw[->>] (21) -- (31);
		\node[circle,draw=none, fill=none, inner sep=0pt, minimum size=5pt] at (1,0.5) {$\odot$};
		\node[circle,draw=none, fill=none, inner sep=0pt, minimum size=5pt] at (4.4,0.0) {.};
		\end{tikzpicture}
\]
\end{construction}


\section{Enumerating transfer systems for $D_{p^n}$}\label{sec:dihedral}

The eventual goal of this paper is to provide a recursive formula for transfer systems on $[1] \times [n]$ which we realize as the subgroup lattice of $G = C_{qp^n}$ for $p,q$ distinct primes. We will begin with a slightly more tame count, that of \emph{liftable transfer systems} of $[1] \times [n]$ as we now define.
\begin{defn}\label{defn:liftable}
A transfer system $\RR$ on $[1] \times [n]$ is \emph{liftable} if $\RR$ additionally satisfies:
\begin{equation*}
	\text{If $(1,i)\, \RR\, (1, j)$ for $i < j$ then $(0,i)\, \RR\, (1,i)$.}\tag{\emph{L}}
\end{equation*}
We write $L(n)$ to denote the collection of liftable transfer systems on $[1] \times [n]$.
\end{defn}

The relevance of liftable transfer systems in the general machinery of $N_\infty$ operads is provided by the following key result of~\cite{bmometa}. 

\begin{prop}[{\cite[\S 4]{bmometa}}]\label{prop:whatarelift}
Let $n \geqslant 0$. Then there is a bijection between liftable transfer systems on $[1] \times [n]$ and transfer systems for the group $G = D_{p^n}$.
\end{prop}

The link between liftable transfer systems and dihedral transfer systems arises via the lattice isomorphism $\Sub(D_{p^n})/D_{p^n} \cong [1] \times [n]$. (Here the quotient is with respect to the conjugation action of $D_{p^n}$ on its subgroups.) Once we have given a recursive formula for $L(n)$, we will, in the following section, exploit the involution on transfer systems as uncovered in~\cite{fooqw} to prove a recursion for $\mathsf{Tr}([1] \times [n])$.

As a warmup, let us see how the liftable condition $(L)$ allows us to enumerate the number of saturated transfer systems in $L(n)$. Recall that a transfer system is said to be \emph{saturated} if it satisfies 2-out-of-3.

	\begin{prop}
There are $(n+2)2^n$ saturated transfer systems for $D_{p^n}$.
	\end{prop}
	
	\begin{proof}
	In~\cite{hmoo} it was proved that there are $2^n$ saturated transfer systems on $[n]$. Consider the bottom row of $[1]\times[n]$ given by the coordinates $(0,i)$, then we can freely pick any saturated transfer system to fill this in. Let $j$ be the maximal number such that $(0,j) \, \RR \, (1,j)$. Clearly there are $n+2$ such choices (including the possibility of no such $j$ existing). We then claim that this data uniquely determines a saturated transfer system on $[1] \times [n]$ which moreover satisfies $(L)$.
	
	We begin by observing that if $j',j'' > j$, we cannot have $(1 , j') \, \RR \, (1 , j'')$. If this were the case then by $(L)$ we would be forced to have $(0,j') \, \RR \, (1,j')$ which contradicts the maximality of $j$. Instead consider the existence of $(1,j) \, \RR \, (1 , j+1)$. By restriction we must have $(0,j) \, \RR \, (1, j+1)$ as well as $(0,j) \, \RR \, (0, j+1)$. As saturated transfer systems satisfy two-out-of-three, it would imply that we moreover have $(0,j+1) \, \RR \, (1,j+1)$, once again contradicting the maximality of $j$. As such, everything to the right of $j$ on the top row of $[1] \times [n]$ must be empty.
	
	As we have $(0,j) \, \RR \, (1,j)$, by restriction we also have $(0 , \ell) \, \RR \, (1 , \ell)$ for all $0 \leqslant \ell \leqslant j$. By the saturation condition, this forces the top row to the left of $j$ to be identical to the bottom row.
	
	Assembling everything, it follows that there are $(n+2)2^n$ possibilities as claimed. 
	\end{proof}

\begin{remark}
Saturated transfer systems have a simple description as those transfer systems which satisfy the 2-out-of-3 property, however they are extremely important in the realm of commutative equivariant homotopy theory. Indeed, they are related to the equivariant linear isometry operads. Every linear isometry operad arises from a saturated transfer system, but this relation need not be bijective.  We refer the reader to \cite{macbrough2023equivariant, Rubin2} for more details.
\end{remark}

We will use the strategy outlined in \autoref{sec:strategy} to give a recursion for $L(n)$. In particular, we begin by considering the partition of $L(n)$ as
\[
L(n) = \coprod_{(a,b) \in [1] \times [n]} L(n,(a,b))
\]
where $L(n,(a,b))$ is the set of liftable transfer systems such that $(a,b)$ is the minimal fibrant element (i.e., $(a,b)$ is minimal such that $(a,b)\, \RR\, (1,n)$).

We will find it advantageous to further split up the elements $L(n,(a,b))$ based on further partitioning properties of transfer systems.

\begin{defn}
Let $\RR$ be a transfer system on $[1] \times [n]$. We say that $(a,b) \in [1] \times [n]$ is:
\begin{itemize}
	\item \emph{Stationary} if $a=1$ and there exists no $d > b \geqslant c$ with $(1,c)\, \RR\, (1,d)$.
	\item \emph{Extendable} if $a=0$ and $(a,b)\, \RR\, (0,n)$.
\end{itemize}
\end{defn}

It is clear that any transfer system possesses at least one stationary and at least one extendable element, namely $(1,n)$ and $(0,n)$ respectively.  We now have a different partition of $L(n)$ as
\[
L(n) = \coprod_{1 \leqslant k \leqslant n+1} L(n,k)
\]
where $L(n,k)$ is the collection of transfer systems such that exactly $k$ elements are stationary. We then further refine each $L(n,k)$ as
\[
L(n,k) = \coprod_{1 \leqslant \ell \leqslant n+1} L(n,k,\ell)
\]
where $L(n,k,\ell) \subseteq L(n,k)$ is the subset of those transfer system where exactly $\ell$ elements are extendable.

Finally, we can define $L(n,k,\ell,(a,b))$ to be $L(n,k,\ell) \cap L(n,(a,b))$. We provide a full description of these transfer systems in the following definition.

\begin{defn}
Let $n \geqslant 0$, $1 \leqslant k, \ell \leqslant n+1, (a,b) \in [1] \times [n]$. We define $L(n,k,\ell,(a,b))$ to be the collection of liftable transfer systems $\RR$ on $[1] \times [n]$ such that:
	\begin{itemize}
	\item $\RR$ has $k$ stationary elements.
	\item $\RR$ has $\ell$ extendable elements.
	\item $(a,b)$ is the minimal fibrant element of $\RR$.
	\end{itemize} 
\end{defn}

In the next collection of propositions we will provide recursions for $L(n,k,\ell,(a,b))$ for varying families of $(a,b) \in [1] \times [n]$.   We require one more definition before continuing to the first case. 

\begin{defn}\label{defn:tam1}
We denote by $\mathrm{Tam}(n) \subset L(n)$ the set of all transfer systems $\RR$ with $(0,n)\, \RR\, (1,n)$ and $\mathrm{Tam}(n,k) =\mathrm{Tam}(n) \cap L(n,k)$.
\end{defn}

We will provide an in-depth exploration of $\mathrm{Tam}(n,k)$ in \autoref{sec:restrict}. In particular, we will explain their namesake (\emph{Tamari intervals}) and provide  explicit formul\ae{} for $|\mathrm{Tam}(n)|$ and $|\mathrm{Tam}(n,k)|$; see \autoref{prop:tamaricounts}.

\begin{prop}\label{prop:sum1}
Let $b >0$. Then
\[
L(n,k,\ell,(0,b)) =  \coprod_{0 \leqslant i \leqslant k} \mathrm{Tam}(b-1,i) \times L(n-b,k-i,\ell,(0,0)).
\]
\end{prop}

\begin{proof}
Let $\RR \in L(n, (0,b))$. Then by restriction to $(0,b)_{\uparrow}^{\mathrm{c}} \cong [1] \times [b-1]$, we obtain a transfer system $\RR_1 \in L(b-1)$, and by restricting to $(0,b)_\uparrow \cong [1] \times [n-b]$, we obtain a transfer system $\RR_2 \in L(n-b)$.  

Since $(0,b)$ is fibrant in $\RR$ and the minimal element of $(0,b)_\uparrow$, we see that $\RR_2 \in L(n-b,(0,0))$. Similarly, since $(0,b)\, \RR\, (1,b)$ by assumption, we also have $(0,b-1)\, \RR\, (1,b-1)$. As such, $(1,b-1)$ is fibrant in $\RR_1$. In particular, we conclude that
\[
\RR_1 \in \mathrm{Tam}(b-1)
\]
and
\[
\RR_2 \in L(n-b,(0,0)).
\]
Conversely, given an arbitrary $\RR_1 \in \mathrm{Tam}(b-1)$ and $\RR_2 \in L(n-b,(0,0))$, we can construct a relation $\RR$ on $[1] \times [n]$ by placing $\RR_2$ to the right of $\RR_1$ (i.e., we reindex $\RR_2$ via a horizontal shift of $b$). We will show that  $\RR_1$ and $\RR_2$ form a restriction closed pair.

Clearly such an $\RR$ is transitive and satisfies the lifting condition. Suppose that $(c,d)\, \RR\, (e,f)$ and $(g,h) \leqslant (e,f)$. If $(e,f) \leqslant (1,b-1)$, then we know that $(c,d)\wedge(g,h)\, \RR\, (g,h)$ by the virtue of $\RR_1$ being a transfer system. 

Else, if $(e,f) \not\leqslant (1,b-1)$, then we must have $(e,f) \geqslant (0,b)$, and since $(c,d)\, \RR\, (e,f)$ and we have no relations from $(0,b)_{\uparrow}^{\mathsf{c}}$ to $(0,b)_\uparrow$, we must also have $(c,d) \geqslant (0,b)$. If $(g,h) \geqslant (0,b)$ as well, then $(c,d) \wedge (g,h)\, \RR\, (g,h)$ from that fact that $\RR_2$ is a transfer system.

The last remaining case is $(0,b) \leqslant (c,d)$ and $(g,h) \leqslant (1,b-1)$, then $(c,d) \wedge (g,h) = (c \wedge g, h)$. If $c \wedge g = g$, then we have $(c \wedge g ,h ) \RR(g,h)$. Otherwise, $c=0$ and $g=1$, so that we just need to ensure that $(0,h)\, \RR\, (1,h)$ fo all $h \leqslant b$, but this follows from that fact that $(0,b-1)$ is fibrant in $\RR_1$.

As such, we have shown that
\[
L(n,(0,b)) = \mathrm{Tam}(b-1) \times L(n-b,(0,0)).
\]

Now we note that $(1,d)$ is stationary in $\RR$ if and only if $d \leqslant b-1$, and $(1,d)$ is stationary in $\RR_1$, or $d \geqslant b$ and $(1,d-b)$ is stationary in $\RR_2$. Furthermore, $(0,d)$ is extendable in $\RR$ if and only if $d \geqslant b$ and $(0,d-b)$ is extendable in $\RR_2$. As such, we can refine the above formula to the desired result.
\end{proof}

\autoref{prop:sum1} is valid whenever $b > 0$. However, we see that the formula is trivial if $b=0$, even though there are liftable transfer systems on $[1] \times [n]$ where the minimal fibrant element is $(0,0)$ (in particular the maximal transfer system is such). The next proposition resolves that case of $b=0$.

\begin{prop}\label{prop:sum2}
\[
L(n,k,\ell,(0,0)) = \coprod_{k-1 \leqslant k' \leqslant n} L(n-1, k', \ell - 1).
\]
\end{prop}

\begin{proof}
Let $\RR \in L(n,k,(0,0))$, and $\RR' \in L(n-1)$ be the restriction to $(0,1)_\uparrow \cong [1] \times [n-1]$. Let $b \leqslant n$ be the maximal such that $(1,0)\, \RR\, (1,b)$. Then $\RR$ is determined by $\RR'$ and $b$. Further, the only constraint that we have on $\RR'$ and $b$ is that if $b \neq 0$ then $(1,b-1)$ must be stationary in $\RR'$.

We then note that an element $(1,d)$ is stationary in $\RR$ if and only if $d \geqslant b$, and either $d=b=0$ or $(1,d-1)$ is stationary in $\RR'$. Conversely, an element $(0,d)$ is extendable if and only if $d=0$ or $(0,d-1)$ is extendable in $\RR'$. The result follows.
\end{proof}

Via \autoref{prop:sum1} and \autoref{prop:sum2}, we have now computed $|L(n,k,\ell,(0,b))|$ for all possible values of $b$. As such we now turn our attention to computing $|L(n,k,\ell,(1,b))|$. The case where $b < n$ is trivial as we now prove.

\begin{prop}\label{prop:sum3}
Let $b < n$, then
\[
L(n,k,\ell,(1,b)) = \varnothing.
\]
\end{prop}

\begin{proof}
Recall that $L(n)$ consists of the liftable transfer systems. Assume that $\RR \in L(n,k,\ell,(1,b))$ then by assumption we have $(1,b)\, \RR\, (1,n)$. However, by the lifting condition, this implies that $(0,b)\, \RR\, (1,b)\, \RR\, (1,n)$, violating the minimality of $(1,b)$. Therefore no such $\RR$ exists.
\end{proof}

The simplicity of the case of $(1,b)$ where $b < n$ is one of the reasons why working with liftable transfer systems is simpler than the general case. Given \autoref{prop:sum3}, it now suffices to compute $L(n,k,\ell,(1,n))$:

\begin{prop}
We have
\[
L(n,k,\ell,(1,n))  = \coprod_{\ell - 1 \leqslant \ell' \leqslant n} L(n-1,k-1, \ell').
\]
\end{prop}

\begin{proof}
Let $\RR \in L(n,k,(1,n))$, and $\RR' \in L(n,-1)$ be the restriction to $[1] \times [n-1]$. Clearly there are no constraints on $\RR'$ except from that it must have $k-1$ stationary elements.

We begin by assuming that there are no relations $(0,d)\, \RR\, (0,n)$ for $d < n$, that is, $\RR$ has a single extendable element. Thus all relations in $\RR$ are contained in $(0,n)^\mathsf{c}_{\uparrow}$ and hence $\RR$ is determined solely by $\RR'$. In particular we have
\[
L(n,k,1,(1,n)) = L(n-1,k-1) = \coprod_{0 \leqslant \ell' \leqslant n} L(n-1,k-1,\ell').
\] 
We now move to the case where $\RR$ has more than a single extendable element. That is, we assume that there exists some $d < n$ with $(0,d)\, \RR\, (0,n)$ for $d < n$. The set of all $d < n$ admitting arrows $(0,d)\, \RR\, (0,n)$ is determined by the maximal such $d$, which we denote $b$. Indeed, if $(0,d)\, \RR\, (0,n)$ then by restriction we also have $(0,d)\, \RR\, (0,b)$, and hence by definition $(0,d) \RR' (0,b)$. Conversely if $(0,d) \RR' (0,b)$, then by transitivity we have $(0,d)\, \RR\, (0,b)\, \RR\, (0,n)$. The only constraint that we have on $b$ is that $(0,b)$ must be extendable in $\RR'$. Note that if $b$ is the $i$-th smallest extendable element in $\RR'$, then $\RR$ has $i+1$ extendable elements. If $\RR'$ has $k$ stationary elements, then $\RR$ has $k+1$ stationary elements. Thus we have shown that
\[
L(n,k,\ell,(1,n)) = \coprod_{\ell - 1 \leqslant \ell' \leqslant n} L(n-1,k-1,\ell').
\] 
as required. We note that this retrieves the above formula when $\ell = 1$.
\end{proof}

We are now in a position to combine the above results into the main theorem of this section which provides an explicit recursive algorithm for computing $|L(n)|$.

\begin{theorem}\label{thm:liftable}
Let $n \geqslant 0$. Then for $1 \leqslant \ell, k \leqslant n+1$ and $(a,b) \in [1] \times [n]$ we have:
	\begin{align*}
	|L(n,k,\ell,(0,0))| = & \sum_{k-1 \leqslant k' \leqslant n} |L(n-1,k',\ell-1)|&\\[10pt]
	|L(n,k,\ell,(1,n))| = & \sum_{\ell-1 \leqslant \ell' \leqslant n} |L(n-1,k-1,\ell')|&\\[10pt]
	|L(n,k,\ell,(0,b))| = & \sum_{0 \leqslant i \leqslant k} |\mathrm{Tam}(b-1,i)| \cdot |L(n-b,k-i,\ell,(0,0))|& (b>0)\\[10pt]
	|L(n,k,\ell,(1,b))| = & {\hspace{1ex}} 0 & (b<n).
	\end{align*}
By convention, set $|L(n,k, \ell,(a,b))| = 0$ when $k, \ell$ are out of range. Then
\[
|L(n)| = \sum_{\substack{1 \leqslant k, \ell \leqslant n+1 \\[4pt] (a,b) \in [1] \times [n]}} |L(n,k,\ell,(a,b))|.
\]
\end{theorem}

\begin{corollary}
Let $n \geqslant 0$. Then the number of homotopically distinct $N_\infty$ operads for $D_{p^n}$ is given by the recursion in \autoref{thm:liftable}.
\end{corollary}

\begin{remark}\label{rem:onlylifthere}
Note that in \autoref{thm:liftable}, it is only the computation $|L(n,k,\ell,(1,b))| = 0$ for $b < n$ where we use the fact that we are working with liftable transfer systems. This will be relevant in the next section we consider transfer systems for $C_{qp^n}$.
\end{remark}

\begin{computation}\label{comp:tn}
By starting at $m=0$, we inductively store the values of $|L (m, k, \ell, (a,b))|$ for all $m \leqslant n$, $1 \leqslant k, \ell \leqslant m+1$ and $(a,b) \in [1] \times [m]$. The entries for all $m < m'$ allow us to compute the values for $m=m'$ via \autoref{thm:liftable} and \autoref{prop:tamaricounts}; some small values are recorded in \autoref{tab:tncomputation}, and large scale structure appears at the end of the paper in \autoref{fig:values}.
\begin{table}[h]
\begin{tabular}{|c|c|}
\hline
\rowcolor[HTML]{C0C0C0} 
\textbf{$n$} & \textbf{$|L(n)|$} \\ \hline
0            & 2                 \\ \hline
1            & 9                 \\ \hline
2            & 56                \\ \hline
3            & 416               \\ \hline
4            & 3457              \\ \hline
5            & 31063             \\ \hline
6            & 295834            \\ \hline
7            & 2948082           \\ \hline
8            & 30471080          \\ \hline
9            & 324580196         \\ \hline
10           & 3546142551        \\ \hline
\end{tabular}
\vspace{3mm}
\caption{Values of $|L(n)|$, the number of transfer systems for $D_{p^n}$ ($p\ne 2$), computed via \autoref{thm:liftable}.}\label{tab:tncomputation}
\end{table}
\end{computation}

\begin{remark}
We leave open the challenge of finding a closed formula for $|L(n)|$. One might begin with a five-variable generating function and the functional equation implied by \autoref{thm:liftable}. Of course, a bijective enumeration would be even more desirable.
\end{remark}
\begin{remark}
In \autoref{prop:max_asymp} we deduce a nontrivial lower bound on the asymptotics of $|L(n)|$.
\end{remark}

\section{Enumerating transfer systems for $C_{qp^n}$}\label{sec:cyclic}

In the previous section we enumerated the liftable transfer systems on $[1] \times [n]$.  We will now move to the general case. For clarity we will write $T(n)$ for the collection of all transfer systems $[1] \times [n]$ (i.e., $L(n) \subseteq T(n)$). We will employ the same strategy as before, and $T(n,k, \ell, (a,b))$ will be as before (but without the lifting condition).

From \autoref{rem:onlylifthere} we see that we need only consider how to resolve the case of $T(n,k,\ell,(1,b))$ for $b < n$, and in the other cases the recursion for $T(n,k, \ell, (a,b))$ is the same as the recursion for $L(n,k, \ell, (a,b))$. To resolve this remaining case, we will use a powerful observation regarding the duality of transfer systems as described in~\cite{fooqw} which we now recall. For a transfer system $\RR$ on a finite lattice $L$, we write
\[
  \mathcal{E}(\RR) = \{ (z,y) \mid \text{ there exists } x \in L \text{ such that } z \leqslant  x < y \text{ and } x\, \RR\, y  \}
\]
for the \emph{downward closure of $\RR$}.

We recall that a lattice $L$ admits a \emph{self duality} if there exists a bijection $\nabla \colon L \to L$ such that $ x \leqslant y$ if and only if $y^\nabla \leqslant x^\nabla $ for all $x,y \in L$.

\begin{defn}
Let $L$ be a finite lattice admitting a self duality $\nabla$ and let $\RR$ be a transfer system on $L$. The \emph{dual} of $\RR$ is the transfer system $\RR^\ast$ defined as
\[
\RR^\ast = ((\mathcal{E}(\RR)^{op})^\nabla)^{\mathsf{c}}.
\]
\end{defn}

From~\cite[Theorem 4.21]{fooqw} we have that $(\RR^\ast)^\ast = \RR$ so this provides an involution on the set $\mathsf{Tr}(L)$. In the case that $L = [1] \times [n]$, we use the canonical duality given by $(a,b) \xmapsto{\nabla} (1-a,n-b)$. In particular we have $(a,b) \, \RR \, (c,d)$ if and only if $(1-c,n-d) \, \cancel{\mathcal{E}(\RR)}  \, (1-a,n-b)$.

\begin{prop}\label{prop:oddoneout}
The duality on $\mathsf{Tr}([1] \times [n])$ restricts to a duality 
\[
T(n,k,\ell,(a,b)) \longleftrightarrow T(n,\ell,k,(1-a,n-b)).
\]
This duality does not preserve the property of being liftable.
\end{prop}

\begin{proof}
Let $\RR \in T(n,k,\ell,(a,b))$. We first observe that $(1-a,n-b)$ is the minimal fibrant element in $\RR^\ast$. Indeed, we have $(a,b)$ is the minimal element such that $(a,b)\, \RR\, (1,n)$. From the definition, we see that $(1-a,n-b)\, \RR^\ast\, (1,n)$ if and only if $(0,0)\, \cancel{\mathcal{E}(\RR)}\, (a,b)$. Assume that we have $(0,0)\, \mathcal{E}(\RR) \, (a,b)$, then this implies the existence of some $(i,j) < (a,b)$ such that $(i,j)\, \RR\, (a,b)$, but by transitivity we would get $(i,j)\, \RR\, (a,b)\, \RR\, (1,n)$, contradicting the minimality of $(a,b)$. The minimality of $(1-a,n-b)$ as a fibrant element of $\RR^\ast$ is afforded by the minimality of $(a,b)$ in $\RR$.

We will now explore the duality between the extendable and stationary elements. Let $y \geqslant 1$, we will show that $(1,y-1)$ is stationary in $\RR$ if and only if $(0,n-y)$ is extendable in $\RR^\ast$. Again by definition, this happens if and only if $(1,0)\, \cancel{\mathcal{E}(\RR)}\, (1,y)$.  Assume that $(1,0)\, \mathcal{E}(\RR)\, (1,y)$, then this implies the existence of some $0 \leqslant z < y$ with $(1,z)\, \RR\, (1,y)$. However, we have assumed that $y-1$ is stationary, so no such $z$ can exist.

The remaining case of $y=0$ is covered by the fact that $(1,n)$ is always stationary and $(0,n)$ is always extendable.
\end{proof}

Combining \autoref{prop:oddoneout}, \autoref{thm:liftable} and \autoref{rem:onlylifthere} we arrive at our desired result.

\begin{theorem}\label{thm:nonliftable}
Let $n \geqslant 0$. Then for $1 \leqslant \ell, k \leqslant n+1$ and $(a,b) \in [1] \times [n]$ we have:
	\begin{align*}
	|T(n,k,\ell,(0,0))| = & \sum_{k-1 \leqslant k' \leqslant n} |T(n-1,k',\ell-1)|&\\[10pt]
	|T(n,k,\ell,(0,b))| = & \sum_{0 \leqslant i \leqslant k} |\mathrm{Tam}(b-1,i)| \cdot |T(n-b,k-i,\ell,(0,0))|& (b>0)\\[10pt]
	|T(n,k,\ell,(1,b))| = & {\hspace{1ex}} |T(n,\ell,k,(0,n-b))|. &
	\end{align*}	
By convention, set $|T(n,k, \ell,(a,b))| = 0$ when $k, \ell$ are out of range. Then
\[
|T(n)| = \sum_{\substack{1 \leqslant k, \ell \leqslant n+1 \\[4pt] (a,b) \in [1] \times [n]}} |T(n,k,\ell,(a,b))|.
\]
\end{theorem}

\begin{corollary}
Let $n \geqslant 0$. Then the number of homotopically distinct $N_\infty$ operads for $C_{qp^n}$ is given by the recursion in \autoref{thm:nonliftable}.
\end{corollary}

\begin{computation}\label{comp:sn}
Following the methods of \autoref{comp:tn} we are once again able to efficiently compute values of $|T(n)|$; for small values of $n$, these are presented in \autoref{tab:sncomputation}, and larger scale behavior appears at the end of the paper in \autoref{fig:values}.
\begin{table}[h]
\begin{tabular}{|c|c|}
\hline
\rowcolor[HTML]{C0C0C0} 
\textbf{$n$} & \textbf{$|T(n)|$} \\ \hline
0            & 2                 \\ \hline
1            & 10                 \\ \hline
2            & 68                \\ \hline
3            & 544               \\ \hline
4            & 4828              \\ \hline
5            & 46124             \\ \hline
6            & 465932            \\ \hline
7            & 4919062           \\ \hline
8            & 53832832          \\ \hline
9            & 607000122         \\ \hline
10           & 7019272236        \\ \hline
\end{tabular}
\vspace{3mm}
\caption{Values of $|T(n)|$, the number of transfer systems for $C_{qp^n}$, computed via \autoref{thm:nonliftable}.}\label{tab:sncomputation}
\end{table}
\end{computation}

\begin{remark}
The strategy outlined \autoref{sec:strategy} is extremely general, and one could envisage running the machine for other fundamental lattices such as $[m] \times [n]$ and $[1]^n$. What is evident from the discussions in \autoref{sec:dihedral} and \autoref{sec:cyclic} is that the power of the computation can be improved by understanding the inherent structures of transfer systems and how they arise (e.g., the existence of suitable liftable families and dualities).
\end{remark}

\begin{remark}
As with $|L(n)|$, it would be desirable to derive a closed formula for $|T(n)|$ via generating function or bijective techniques.
\end{remark}

\section{Restricted Tamari intervals}\label{sec:restrict}

In this section we will study the term $\mathrm{Tam}(n,k)$ which appears in \autoref{prop:sum1}, \autoref{thm:liftable}, and \autoref{thm:nonliftable}.

We begin by recalling from \autoref{defn:tam1} that we denote by $\mathrm{Tam}(n)$ the set of all transfer systems $\RR$ with $(0,n)\, \RR\, (1,n)$ and $\mathrm{Tam}(n,k) =\mathrm{Tam}(n) \cap L(n,k)$. 

Recall from~\cite[Proposition 2.23, Proposition 4.17]{boor} that $\mathsf{Tr}([n])$ admits a lattice bijection to the Tamari lattice. As such for $\RR$ and $\RR'$ transfer systems on $[n]$, we say that $\RR \leqslant \RR'$ is a \emph{Tamari interval}. The next lemma justifies our choice of notation for $\mathrm{Tam}(n)$. We note that if $\RR$ is a transfer system on $[n]$, then we --- analogously to the $[1] \times [n]$ case --- say that $b \in [n]$ is \emph{stationary} for $\RR$ if there exists no $d > b \geqslant c$ with $(1,c)\, \RR\, (1,d)$.

\begin{lemma}
Let $n \geq 0$. Then $\mathrm{Tam}(n)$ is in bijection with Tamari intervals in $\mathsf{Tr}[n]$. In particular $\mathrm{Tam}(n,k)$ is in bijection with the Tamari intervals $\RR \leqslant \RR'$ in $\mathsf{Tr}[n]$ such that $\RR$ has $k$ stationary elements.
\end{lemma}

\begin{proof}
By restriction, a transfer system in $\mathrm{Tam}(n)$ has all vertical relations $(0,a)\, \RR$ $(1,a)$ for $0 \leqslant a \leqslant n$. We have the diagonal $(0,b)\, \RR\, (1,c)$ if and only if $(0,b)\, \RR\, (0,c)$. Indeed, the forward implication follows by pullback closure, and the reverse implication is given by composition with $(0,c)\, \RR\, (1,c)$. Therefore, such an $\RR$ restricted to the bottom row determines all diagonals. In particular $\RR$ is determined uniquely by its restriction to the top and bottom, and is therefore the data of a Tamari interval.
\end{proof}

From work of Chapoton~\cite{chapoton}, we know that the cardinality of the set of Tamari intervals for $[n]$ has closed form given by
\[
|\mathrm{Tam}(n)| = \dfrac{2}{(n+1)(n+2)}\binom{4n+5}{n}.
\]
(We warn the reader that our indexing conventions differ from those of \cite{chapoton}). The following proposition provides a closed formula for $|\mathrm{Tam}(n,k)|$.

\begin{prop}\label{prop:tamaricounts}
Let $n \geqslant 0$ and $1 \leqslant k \leqslant n+1$. Then

\[
|\mathrm{Tam}(n,k)| =  \dfrac{2(2k + 1)!(4n - 2k + 3)!}{(k - 1)!(k + 1)!(n - k + 1)!(3n - 
       k + 4)!} 
\]
\end{prop}

\begin{proof}
An explicit bijection between transfer systems on $[n]$ and rooted binary trees with $(n+1)$ internal nodes was constructed in~\cite{bbr}. One can check that under this bijection that the number of stationary elements of a transfer system $\RR$ on $[n]$ is equal to the number of elements on the corresponding binary tree whose path from the root is a straight line to the left.

Using the classical bijection between binary trees with $n+1$ nodes and Dyck paths of length $2n+2$, we see that the number of stationary elements is also given by one less than the number of times the corresponding Dyck path touches the horizontal axis.

The result then follows from~\cite[Corollary 11]{BMFPR} and the first remark following it.
\end{proof}

\begin{table}[h]
\renewcommand\arraystretch{1.5}
\begin{tabular}{|
>{\columncolor[HTML]{C0C0C0}}c |c|c|c|c|c|c|c|}
\hline
\diagbox[width=\dimexpr \textwidth/10\relax, height=1.1cm]{$n$}{$k$}& \cellcolor[HTML]{C0C0C0}{1} & \cellcolor[HTML]{C0C0C0}{2} & \cellcolor[HTML]{C0C0C0}{3} & \cellcolor[HTML]{C0C0C0}{4} & \cellcolor[HTML]{C0C0C0}{5} & \cellcolor[HTML]{C0C0C0}{6} & \cellcolor[HTML]{C0C0C0}{7} \\ \hline
{0}                  & 1                                  & 0                                  & 0                                  & 0                                  & 0                                  & 0                                  & 0                                  \\ \hline
{1}                  & 1                                  & 2                                  & 0                                  & 0                                  & 0                                  & 0                                  & 0                                  \\ \hline
{2}                  & 3                                  & 5                                  & 5                                  & 0                                  & 0                                  & 0                                  & 0                                  \\ \hline
{3}                  & 13                                 & 20                                 & 21                                 & 14                                 & 0                                  & 0                                  & 0                                  \\ \hline
{4}                  & 68                                 & 100                                & 105                                & 84                                 & 42                                 & 0                                  & 0                                  \\ \hline
{5}                  & 399                                & 570                                & 595                                & 504                                & 330                                & 132                                & 0                                  \\ \hline
{6}                  & 2530                               & 3542                               & 3675                               & 3192                               & 2310                               & 1287                               & 429                                \\ \hline
\end{tabular}
\vspace{3mm}
\caption{Values of $|\mathrm{Tam}(n,k)|$ for small values of $n$ and $k$ using \autoref{prop:tamaricounts}.}\label{tab:master}
\end{table}

\begin{remark}
The values appearing in \autoref{tab:master} make up some very well known number sequences as we now outline. The triangular array $|\mathrm{Tam}(n,k)|$ appears in a reindexed form in~\cite{brown} where one enumerates triangulations of $k+3$-gons with $n+1$-internal vertices. In particular we have:
\begin{itemize}
\item The major diagonal where $k=n+1$ retrieves the Catalan numbers.
\item The sum of the row $|\mathrm{Tam}(n,-)|$ computes the number of Tamari intervals for $[n]$.
\item The column $|\mathrm{Tam}(-,1)|$ is the number of Tamari intervals for $[n-1]$.
\item The column $|\mathrm{Tam}(-,k)|$ is the number of triangulations of a $(k+2)$-gon with $n$ internal nodes.
\end{itemize}
\end{remark}

\section{Maximally extendable $D_{p^n}$ transfer systems and Schr\"{o}der numbers}\label{sec:max1}

In the remainder of this paper, we will explore a certain type of transfer system on $[1] \times [n]$ in the liftable and nonliftable cases, the enumeration of which retrieve interesting number sequences. We begin by introducing the class that we are interested in.

\begin{defn}
A (liftable) transfer system $\RR$ for $[1] \times [n]$ is \emph{maximally extendable} if it has $n+1$ extendable elements.
\end{defn}

\begin{remark}
In a more intuitive description, a transfer system is maximally extendable if and only if its restriction to the bottom row is the maximal transfer system on $[n]$. As such, for $G = D_{p^n}$ or $C_{qp^n}$ (in the liftable and general scenarios, respectively) these are the transfer systems associated with $G$-$N_\infty$ operads restricting to genuine $G$-$E_\infty$ operads on $C_{p^n}\leqslant G$.
\end{remark}

We will once again begin with the case of liftable transfer systems.

\begin{defn}
Let $n \geqslant 0$. Then $\mathfrak{S}_n$, the $n^\mathrm{th}$ \emph{(large) Schr\"{o}der number}, is the number of lattice paths from the southwest corner $(0,0)$ of an $n \times n$ grid to the northeast corner $(n,n)$ using single steps north $(0,1)$, northeast $(1,1)$, or east $(1,0)$, that do not rise above the $SW$-$NE$ diagonal. (Such paths are called \emph{royal ($n$-)paths}.)
\end{defn}

\begin{prop}[\cite{MR3662973}]\label{prop:schrodercount}
If $n = 0$ then $\mathfrak{S}_n = 1$. For $n\ge 1$, the large Schr\"{o}der numbers satisfy the recurrence
\[
  \mathfrak S_n = 3\mathfrak S_{n-1} + \sum_{k=1}^{n-2} \mathfrak S_k \mathfrak S_{n-k-1}.
\]
Additionally, for $n \geqslant 1$ we have
\[
\mathfrak{S}_n = \sum_{1 \leqslant k \leqslant n} \mathrm{Nar}(n,k) \cdot 2^k = \sum_{1 \leqslant k \leqslant n} \dfrac{1}{n} \binom{n}{k} \binom{n}{k-1} \cdot 2^k
\]
where $\mathrm{Nar}(n,k)$ is the $(n,k)$-th Narayana number.
\end{prop}

\begin{computation}\label{com:schroder}
\autoref{tab:shroder} provides a calculation of $\mathfrak{S}_n$ for small values of $n$.
\begin{table}[h]
\begin{tabular}{|c|c|}
\hline
\rowcolor[HTML]{C0C0C0} 
\textbf{$n$} & \textbf{$\mathfrak{S}_n$} \\ \hline
0            & 1                 \\ \hline
1            & 2                 \\ \hline
2            & 6                 \\ \hline
3            & 22                \\ \hline
4            & 90               \\ \hline
5            & 394              \\ \hline
6            & 1806             \\ \hline
7            & 8558            \\ \hline
8            & 41586           \\ \hline
9            & 206098          \\ \hline
10            & 1037718         \\ \hline
\end{tabular}
\vspace{3mm}
\caption{Values of $\mathfrak{S}_n$ computed via \autoref{prop:schrodercount}.}\label{tab:shroder}
\end{table}
\end{computation}

We can produce refined statistics on royal paths by tracking the number of diagonal returns each path has, using the convention that a northeast step on the diagonal counts as a return. Let $\mathfrak S_n(k)$ denote the number of royal $n$-paths with $k$ returns to the diagonal; call these the \emph{refined Schr\"oder numbers}.

\begin{prop}[{\cite{oeis:sch}}]\label{prop:refined_schroeder}
The refined Schr\"oder numbers take the values
\[
  \mathfrak S_n(k) =
  \begin{cases}
    0 &\text{if }n<k,\\
    2^n &\text{if }n=k,\\
    2^k\frac{k}{n-k}\sum_{p=1}^{n-k}\binom{n-k}{p}\binom{n-1+p}{p-1}&\text{if }n>k.
  \end{cases}
\]
Furthermore,
\[
  \sum_{k=1}^n \mathfrak S_n(k) = \mathfrak S_n.
\]
\end{prop}

\begin{computation}\label{com:schroder}
\autoref{tab:refined_shroder} provides a calculation of $\mathfrak{S}_n(k)$ for small values of $n$ and $k$.
\begin{table}[h]
\renewcommand\arraystretch{1.5}
\begin{tabular}{|
>{\columncolor[HTML]{C0C0C0}}c |c|c|c|c|c|c|}
\hline
\diagbox[width=\dimexpr \textwidth/10\relax, height=1.1cm]{$n$}{$k$}& \cellcolor[HTML]{C0C0C0}{1} & \cellcolor[HTML]{C0C0C0}{2} & \cellcolor[HTML]{C0C0C0}{3} & \cellcolor[HTML]{C0C0C0}{4} & \cellcolor[HTML]{C0C0C0}{5} & \cellcolor[HTML]{C0C0C0}{6}  \\ \hline
{1}                  & 2                                  & 0                                  & 0                                  & 0                                  & 0                                  & 0                                                                    \\ \hline
{2}                  & 2                                  & 4                                  & 0                                  & 0                                  & 0                                  & 0                                                                    \\ \hline
{3}                  & 6                                 & 8                                 & 8                                 & 0                                 & 0                                  & 0                                                                    \\ \hline
{4}                  & 22                                 & 28                                & 24                                & 16                                 & 0                                 & 0                                                                    \\ \hline
{5}                  & 90                                & 112                                & 96                                & 64                                & 32                                & 0                                                                  \\ \hline
{6}                  & 394                               & 484                               & 416                               & 288                               & 160                               & 64                                                               \\ \hline
\end{tabular}
\vspace{3mm}
\caption{Values of $\mathfrak{S}_n(k)$ computed via \autoref{prop:refined_schroeder}.}\label{tab:refined_shroder}
\end{table}
\end{computation}

While the above result is standard, we were unable to find a proof of the following recurrence in the literature. It will be extremely useful in linking large Schr\"oder numbers with transfer systems.

\begin{lemma}\label{lemma:royal_recurrence}
The refined Schr\"oder numbers satisfy the recurrence relation
\[
  \mathfrak S_n(k) = 2\mathfrak S_n(k-1) + \sum_{p=k}^n \mathfrak S_{n-1}(p)
\]
for $1\le k\le n$.
\end{lemma}
\begin{proof}
Write $E$ for east steps, $N$ for north steps, and $D$ for diagonal steps. Given a royal $n$-path, consider the location of its first diagonal return. If this return is in position $(1,1)$, then the path begins $EN$ or $D$ and the rest of the path may be filled in by any royal $(n-1)$-path with $k-1$ returns. This accounts for the $2\mathfrak S_n(k-1)$ term. If the first return is in position $(r,r)$ with $r>1$, then we may write our path as $EPNQ$ where $P$ is a royal $(r-1)$-path and $Q$ is a royal $(n-r)$-path. The path $PQ$ is then a royal $(n-1)$-path with at least $k$ diagonal returns. This decomposition is unique and each such royal $(n-1)$-path arises in this fashion, so this accounts for the term $\sum_{p=k}^n \mathfrak S_{n-1}(p)$.
\end{proof}

We are now prepared to demonstrate the connection between (refined) large Schr\"oder numbers and maximally extendable transfer systems on $D_{p^n}$.

\begin{theorem}\label{thm:schroder}
Let $n \geqslant 0$. Then the number of maximally extendable transfer systems in $L(n)$ is given by $\mathfrak{S}_{n+1}$, and the number of maximally extendable transfer systems in $L(n)$ with exactly $k$ stationary nodes is $\mathfrak S_{n+1}(k)$.
\end{theorem}

\begin{proof}
It suffices to prove the second statement since the refined Schr\"oder numbers $\mathfrak S_{n+1}(k)$ add up to $\mathfrak S_{n+1}$. Write $\mathfrak m(n,k) := |L(n,k,n+1)|$ for the number of maximally exetndable transfer systems on $D_{p^n}$ with exactly $k$ stationary nodes. Clearly $\mathfrak m(0,1) = 2 = \mathfrak S_1(1)$, so it suffices to prove that $\mathfrak m(n,k)$ satisfies the recurrence of \autoref{lemma:royal_recurrence}, \emph{i.e.},
\[
  \mathfrak m(n,k) = 2\mathfrak m(n,k-1) + \sum_{p=k}^n \mathfrak m(n-1,p).
\]
This is a direct consequence of \autoref{thm:liftable} specialized to the case $\ell = n+1$.
\end{proof}

\begin{remark}
In \cite[Exercise 6.39]{stanley:ec2}, Stanley provides a list of nineteen classical structures counted by Schr\"oder numbers. The challenge of finding an explicit bijection between maximally extendable transfer systems on $D_{p^n}$ and one of these structures remains open.
\end{remark}

\begin{prop}\label{prop:max_asymp}
Let $L^{\mathrm{max}}(n)$ denote the collection of maximally extendable liftable transfer systems on $[1]\times [n]$. The asymptotics of $|L^{\mathrm{max}}(n)|$ satisfy
\[
  |L^{\mathrm{max}}(n)|\sim C\frac{(3+\sqrt{8})^n}{n^{3/2}}
\]
where
\[
  C = \frac{\sqrt{2}(3+2\sqrt{2})}{2\sqrt{\pi(-4+3\sqrt{2})}}\approx 4.720408926.
\]
This provides a lower bound on the asymptotic behavior of $|L(n)|$.
\end{prop}
\begin{proof}
This follows directly from \autoref{thm:schroder} and the asymptotics of the large Schr\"oder numbers (see for instance Exercise 12 on p.539 of \cite{knuth}).
\end{proof}

We now shift our attention back to all transfer systems on $[1] \times [n]$, and consider the maximally extendable transfer systems here. Recall from \autoref{prop:oddoneout} that there is a duality
\[
T(n,k,\ell,(a,b)) \longleftrightarrow T(n,\ell,k,(1-a,n-b)).
\]
In particular, by summing up over all elements $(a,b) \in [1] \times [n]$ we obtain a duality
\[
T(n,k,\ell) \longleftrightarrow T(n,\ell,k).
\]

\begin{defn}
A transfer system $\RR$ for $[1] \times [n]$ is \emph{maximally stationary} if it has $n+1$ stationary elements.
\end{defn}

\begin{remark}
In a more intuitive description, a transfer system is maximally stationary if and only if its restriction to the top row is the minimal transfer system on $[n]$.
\end{remark}

\begin{corollary}\label{cor:maxswap}
Let $n \geqslant 0$. Then the number of maximally extendable transfer systems in $T(n)$ is the same as the number of maximally stationary transfer systems in $T(n)$ (and indeed the same as the number of maximally stationary transfer systems in $L(n)$).
\end{corollary}

\begin{defn}
Let $n \geqslant 1$. Then $\mathfrak{A}_n$ is the number of rooted subtrees in rooted planar tree with $n$ nodes.
\end{defn}

\begin{remark}
Rooted subtrees of a fixed planar tree are in bijection with nonempty antichains in the same tree (with the canonical partial order in which the root is minimal and edges are covering relations). Rooted subtrees are analyzed in \cite{ruskey} while \cite{klazar} studies the same problem in the language of antichains.
\end{remark}

\begin{prop}[\cite{ruskey,klazar}]\label{prop:antichain}
Let $n \geqslant 1$. Then
\[
\mathfrak{A}_n = \dfrac{\displaystyle{\sum_{0 \leqslant i < n}}  \binom{2i+1}{i} \binom{2n-1}{n-i-1}}{2n-1}.
\]
Moreover $\mathfrak{A}_n$ satisfies the recurrence $\mathfrak{A}_1 = 1$ and
\[
\mathfrak{A}_{n} = \sum_{j=1}^{n-1} \mathfrak{A}_{n-j} \mathfrak{A_{j}} + \mathfrak{A}_{n-j} \mathrm{Cat}(j-1) 
\]
where $\mathrm{Cat}(j)$ is the $j$-th Catalan number.

\end{prop}

\begin{computation}\label{com:antichain}
\autoref{tab:antichain} provides a calculation of $\mathfrak{A}_n$ for small values of $n$.
\begin{table}[h]
\begin{tabular}{|c|c|}
\hline
\rowcolor[HTML]{C0C0C0} 
\textbf{$n$} & \textbf{$\mathfrak{A}_n$} \\ \hline
1            & 1                 \\ \hline
2            & 2                 \\ \hline
3            & 7                 \\ \hline
4            & 29                \\ \hline
5            & 131               \\ \hline
6            & 625              \\ \hline
7            & 3099             \\ \hline
8            & 15818            \\ \hline
9            &  82595          \\ \hline
10            & 439259          \\ \hline
\end{tabular}
\vspace{3mm}
\caption{Values of $\mathfrak{A}_n$ computed via \autoref{prop:antichain}.}\label{tab:antichain}
\end{table}
\end{computation}

Numerical experiments lead us to conjecture that maximally extendable transfer systems in $T(n)$ are counted by $\mathfrak A_{n+2}$. By \autoref{cor:maxswap}, we conjecture the same count for maximally stationary transfer systems in $T(n)$ and maximally stationary transfer systems in $L(n)$.

\begin{conjecture}
For $n \geqslant 0$, the number of maximally extendable transfer systems in $T(n)$ is given by $\mathfrak{A}_{n+2}$.
\end{conjecture}

The authors' attempts to prove this conjecture by standard enumerative methods (both bijective and inductive) have been spoiled by the subtle effects of restriction on relations joining the bottom and top rows. Following the argument of \cite{klazar}, we note that it would suffice to decompose the top row $r$ of each maximally extendable transfer system in $T(n)$ into ``primary components'' $r_1,\ldots,r_k$ for which
\[
  w(r) = \prod_{i=1}^k (1+w(r_i))
\]
where $w(s)$ measures the number of maximally extendable transfer systems with top row $s$.

\subsection{Figures}
We conclude by plotting of the behavior of our enumerations related to transfer systems for $D_{p^n}$ and $C_{qp^n}$. All of the data for these figures was produced using code available at \url{https://github.com/bifibrant/recursion/}.

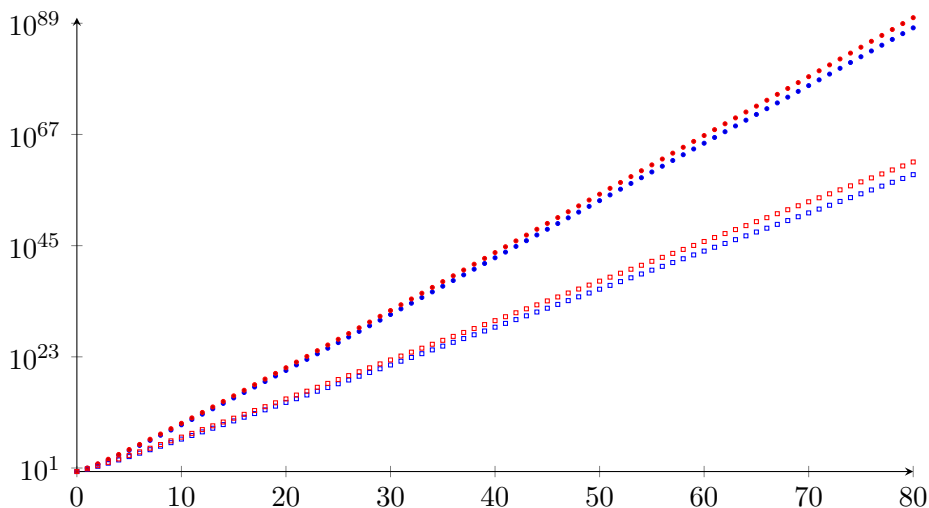
\begin{figure}[h]
\center
\begin{tikzpicture}
\begin{semilogyaxis}[
  width=5in,
  height=3in,
  mark size=0.75pt,
  axis lines=left,
  scaled ticks=false,
  only marks,
]
  \addplot table {L.dat};
  \addplot+ [mark=*] table {T.dat};
  \addplot+ [mark=square,mark options={draw=blue,fill=white}] table {Lmax.dat};
  \addplot+ [mark=square,mark options={draw=red,fill=white}] table {Tmax.dat};
\end{semilogyaxis}
\end{tikzpicture}
\caption{A semilog plot of the values of $|L(n)|$ (blue dots), $|T(n)|$ (red dots), $|L^{\mathrm{max}}(n)|$ (blue squares), and $|T^{\mathrm{max}}(n)|$ (red squares) for $0\leqslant n\leqslant 80$. Here the superscript $\mathrm{max}$ indicates maximally extendable transfer systems.}\label{fig:values}
\end{figure}

The graph in \autoref{fig:values} suggests that $|L(n)|$ and $|T(n)|$ might have subexponential growth rates of a shape similar to the asymptotics of $|L^{\mathrm{max}}(n)|$ from \autoref{prop:max_asymp}.

Finally, in \autoref{fig:ratio} we plot $|L(n)|/|T(n)|$ and observe that the fraction of liftable transfer systems appears to approach $0$.

\begin{figure}[h]
\center
\begin{tikzpicture}
\begin{axis}[
  width=5in,
  height=3in,
  mark size=0.75pt,
  axis lines=left,
  scaled ticks=false,
  only marks,
]
  \addplot table {ratio_LT.dat};
\end{axis}
\end{tikzpicture}
\caption{A (nonlogarithmic) plot of $|L(n)|/|T(n)|$ for $0\leqslant n\leqslant 80$. The semilogarithmic plot appears approximately linear.}\label{fig:ratio}
\end{figure}
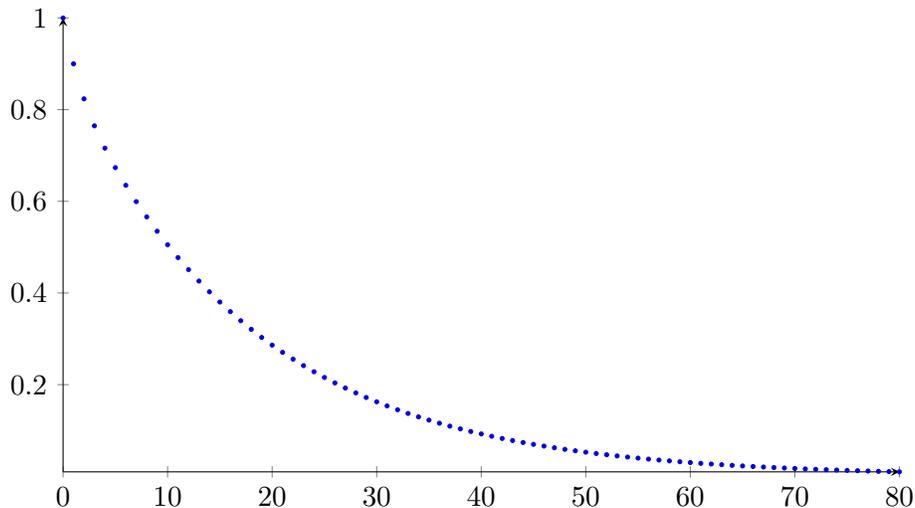



\bibliography{mfgroups}

\newcommand{\etalchar}[1]{$^{#1}$}
\newcommand\OIlabel{}\newcommand\OI[2]{OEIS}
\begin{thebibliography}{BMFPR11}

\bibitem[BBR21]{bbr}
S.~Balchin, D.~Barnes, and C.~Roitzheim.
\newblock {$N_\infty$}-operads and associahedra.
\newblock {\em Pacific J. Math.}, 315(2):285--304, 2021.

\bibitem[BH15]{BlumbergHill}
A.~J. Blumberg and M.~A. Hill.
\newblock Operadic multiplications in equivariant spectra, norms, and
  transfers.
\newblock {\em Adv. Math.}, 285:658--708, 2015.

\bibitem[BHK{\etalchar{+}}23]{bao2023transfer}
Linus Bao, Christy Hazel, Tia Karkos, Alice Kessler, Austin Nicolas, Kyle
  Ormsby, Jeremie Park, Cait Schleff, and Scotty Tilton.
\newblock Transfer systems for rank two elementary abelian groups:
  characteristic functions and matchstick games, 2023.

\bibitem[BMFPR11]{BMFPR}
M.~Bousquet-M\'{e}lou, \'{E}. Fusy, and L.-F. Pr\'{e}ville-Ratelle.
\newblock The number of intervals in the {$m$}-{T}amari lattices.
\newblock {\em Electron. J. Combin.}, 18(2):Paper 31, 26, 2011.

\bibitem[BMO23]{bmometa}
S.~Balchin, E.~MacBrough, and K.~Ormsby.
\newblock Lifting {$N_{\infty}$} operads from conjugacy data.
\newblock {\em Tunis. J. Math.}, 5(3):479--504, 2023.

\bibitem[BOOR23]{boor}
S.~Balchin, K.~Ormsby, A.~M. Osorno, and C.~Roitzheim.
\newblock Model structures on finite total orders.
\newblock {\em Math. Z.}, 304(3):Paper No. 40, 35, 2023.

\bibitem[Bro64]{brown}
W.~G. Brown.
\newblock Enumeration of triangulations of the disk.
\newblock {\em Proc. London Math. Soc. (3)}, 14:746--768, 1964.

\bibitem[Cha07]{chapoton}
F.~Chapoton.
\newblock Sur le nombre d'intervalles dans les treillis de {T}amari.
\newblock {\em S\'{e}m. Lothar. Combin.}, 55:Art. B55f, 18, 2005/07.

\bibitem[FOO{\etalchar{+}}22]{fooqw}
E.~E. Franchere, K.~Ormsby, A.~M. Osorno, W.~Qin, and R.~Waugh.
\newblock Self-duality of the lattice of transfer systems via weak
  factorization systems.
\newblock {\em Homology, Homotopy and Applications}, 24(2):115--134, 2022.

\bibitem[HMOO22]{hmoo}
U.~Hafeez, P.~Marcus, K.~Ormsby, and A.M. Osorno.
\newblock Saturated and linear isometric transfer systems for cyclic groups of
  order {$p^mq^n$}.
\newblock {\em Topology Appl.}, 317:Paper No. 108162, 2022.

\bibitem[Kla97]{klazar}
M.~Klazar.
\newblock Twelve countings with rooted plane trees.
\newblock {\em European J. Combin.}, 18(2):195--210, 1997.

\bibitem[Knu97]{knuth}
D.~E. Knuth.
\newblock {\em The art of computer programming. {V}ol. 1}.
\newblock Addison-Wesley, Reading, MA, 1997.
\newblock Fundamental algorithms, Third edition.

\bibitem[Mac23]{macbrough2023equivariant}
E.~MacBrough.
\newblock Equivariant linear isometries operads over abelian groups, 2023.

\bibitem[\OI22]{oeis:sch}
\OIlabel{{OEIS} Foundation Inc.}
\newblock Entry {A108891} in {T}he {O}n-{L}ine {E}ncyclopedia of {I}nteger
  {S}equences, 2022.
\newblock \url{https://oeis.org/A108891}.

\bibitem[QG17]{MR3662973}
F.~Qi and B.-N. Guo.
\newblock Some explicit and recursive formulas of the large and little
  {S}chr\"{o}der numbers.
\newblock {\em Arab J. Math. Sci.}, 23(2):141--147, 2017.

\bibitem[Rub21]{Rubin2}
Jonathan Rubin.
\newblock Combinatorial {$N_\infty$} operads.
\newblock {\em Algebr. Geom. Topol.}, 21(7):3513--3568, 2021.

\bibitem[Rus81]{ruskey}
F.~Ruskey.
\newblock Listing and counting subtrees of a tree.
\newblock {\em SIAM J. Comput.}, 10(1):141--150, 1981.

\bibitem[Sta99]{stanley:ec2}
R.~P. Stanley.
\newblock {\em Enumerative combinatorics. {V}ol. 2}, volume~62 of {\em
  Cambridge Studies in Advanced Mathematics}.
\newblock Cambridge University Press, Cambridge, 1999.
\newblock With a foreword by Gian-Carlo Rota and appendix 1 by Sergey Fomin.

\end{thebibliography}
\bibliographystyle{alpha}

\end{document}